\documentclass[11pt,reqno]{amsart}
\usepackage{amsmath,amssymb}
\usepackage{mathtools}
\usepackage[margin=1in]{geometry}
\DeclarePairedDelimiter{\ceil}{\lceil}{\rceil}

\newtheorem{thm}{Theorem}[section]
\newtheorem{lem}[thm]{Lemma}
\newtheorem{cor}[thm]{Corollary}
\newtheorem{conj}[thm]{Conjecture}

{\theoremstyle{definition}

\newtheorem{definition}{Definition}

\newtheorem*{question}{Question} }

\newcommand{\wo}{\setminus}

\newcommand{\abs}[1]{\left|{#1}\right|}

\newcommand{\set}[1]{\left\{{#1}\right\}} 
\newcommand{\setof}[2]{\left\{{#1}\,:\,{#2}\right\}}
\newcommand{\of}{\subseteq}

\newcommand{\I}{\mathcal{I}}
\newcommand{\N}{\mathbb{N}}

\newcommand{\al}{q}

\newcommand{\Fmk}[1]{\mathcal{F}^{k,\al;m}_{#1}}

\title{Maximal-clique partitions and the Roller Coaster Conjecture}
\author{Jonathan Cutler}
\address[Jonathan Cutler]{Department of Mathematical Sciences, Montclair State University, Montclair, NJ 07043 USA}
\email{jonathan.cutler@montclair.edu}
\author{Luke Pebody}
\address[Luke Pebody]{London, UK}
\email{luke@pebody.org}

\begin{document}
\begin{abstract}
A graph $G$ is {\em well-covered} if every maximal independent set has the
same cardinality $\al$. Let $i_k(G)$ denote the number of independent sets of
cardinality $k$ in $G$. Brown, Dilcher, and Nowakowski conjectured that
the independence sequence $(i_0(G), i_1(G), \ldots, i_\al(G))$ was unimodal for
any well-ordered graph $G$ with independence number $\al$.  Michael and Traves disproved this conjecture. Instead they posited the so-called ``Roller Coaster" Conjecture: that the terms 
\[
	i_{\ceil{\frac\al2}}(G), i_{\ceil{\frac\al2}+1}(G), \ldots, i_\al(G)
\] 
could be in any specified order for some well-covered graph $G$
with independence number $\al$. Michael and Traves proved the conjecture for $\al<8$ and Matchett extended this to $\al<12$.

In this paper, we prove the Roller Coaster Conjecture using a construction of graphs with a property related to that of having a maximal-clique partition.  In particular, we show,
for all pairs of integers $1\le k<\al$ and positive integers $m$, that there is
a well-covered graph $G$ with independence number $\al$ for which every 
independent set of size $k+1$ is contained in a unique maximal independent set,
but each independent set of size $k$ is contained in at least $m$ distinct
independent sets.  
\end{abstract}

\maketitle
\section{Introduction}\label{S:introduction}

The behavior of the coefficients of the independence polynomial of graphs in various classes has produced many interesting problems.  For a graph $G$, we let $\I(G)$ be the set of independent sets in $G$, i.e., $\I(G)=\setof{I\of V(G)}{E(G[I])=\emptyset}$.  Also, let $\I_k(G)=\setof{I\in \I(G)}{\abs{I}=k}$ and $i_k(G)=\abs{\I_k(G)}$.  The \emph{independence number of $G$} is given by $\alpha(G)=\max \setof{k\in \N}{i_k(G)>0}$.  We let the \emph{independence polynomial of $G$} be the polynomial defined by
\[
	I(G;x)=\sum_{k=0}^{\alpha(G)} i_k(G)x^k.
\]
We refer to $(i_0(G),i_1(G),\ldots,i_{\alpha(G)}(G))$ as the \emph{independence sequence of $G$}.

Natural questions arise when one considers possible orderings of the coefficients of the independence sequence over various classes of graphs.  If one considers the class of all graphs, then Alavi, Erd\H{o}s, Malde, and Schwenk \cite{AEMS} proved that the coefficients can be ordered in any way apart from $i_0(G)=1$.  In particular, they proved the following.  Throughout the paper, we let $[n]=\set{1,2,\ldots,n}$.

\begin{thm}[Alavi, Erd\H{o}s, Malde, Schwenk \cite{AEMS}]
	Given a positive integer $\al$ and a permutation $\pi$ of $[\al]$, there is a graph $G$ with $\alpha(G)=\al$ such that
	\[
		i_{\pi(1)}(G)<i_{\pi(2)}(G)<\cdots<i_{\pi(\al)}(G).
	\]
\end{thm}

A graph $G$ is said to be \emph{well-covered} if every maximal independent set in $G$ has the same size.  Brown, Dilcher, and Nowakowski \cite{BDN} conjectured that the independence sequence of any well-covered graph is unimodal.  This conjecture was disproved by Michael and Traves \cite{MT}.  However, they were able to show the following.

\begin{thm}[Michael, Traves~\cite{MT}]\label{T:mt}
The independence sequence of a well-covered 
graph $G$ with $\alpha(G)=\al$ satisfies
	\[
	\frac{i_0(G)}{\binom{\al}0}\le\frac{i_1(G)}{\binom{\al}1}\le\ldots\le\frac{i_{\al}(G)}{\binom{\al}{\al}}.
	\]
\end{thm}

This implies the following.

\begin{cor}[Michael, Traves \cite{MT}]
	If $G$ is a well-covered graph with $\alpha(G)=\al$, then
	\[
		i_0(G)<i_1(G)<\cdots<i_{\ceil{\al/2}}(G).
	\]
\end{cor}

In addition, Michael and Traves conjectured that the second half of the independence sequence can be ``any-ordered''.  To be precise, they conjectured the following, which has become known as the Roller Coaster Conjecture.

\begin{conj}[Michael, Traves \cite{MT}; Roller Coaster Conjecture]\label{C:rollercoaster}
	Given a positive integer $\al$ and a permutation $\pi$ of $\set{\ceil{\al/2},\ceil{\al/2}+1,\ldots,\al}$, there is a well-covered graph $G$ with $\alpha(G)=\al$ and
	\[
		i_{\pi(\ceil{\al/2})}(G)<i_{\pi(\ceil{\al/2}+1)}(G)<\cdots<i_{\pi(\al)}(G).
	\]
\end{conj}

In addition, Michael and Traves proved the conjecture for $\al\leq 7$.  Matchett \cite{M} was able to prove the Roller Coaster Conjecture for $\al\leq 11$.  He also proved that for sufficiently large $\al$, the last $(.1705)\al$ terms in the independence sequence of some well-covered graph can be any-ordered.  Related work has been done in the context of pure $O$-sequences in order ideals \cite{BMMNZ}.

We will show that a partial converse to Theorem~\ref{T:mt} is true.  Consider the following definition.
\begin{definition}
We say that a polynomial $a_{\al}x^{\al}+\cdots+a_1x$ is an \emph{approximate 
well-covered independence polynomial} if for all real numbers 
$\epsilon>0$, there exists a well-covered graph $G$ of independence 
number $\al$ and a real number $T$ such that for all $1\le k\le \al$, 
\begin{equation}
	\abs{\frac{i_k(G)}T-a_k}<\epsilon.\label{eqn:eps}
\end{equation}
Given such real numbers $T,\epsilon$ and graph $G$, say that $G$ is an \emph{$\epsilon$-certificate for $a_{\al}x^{\al}+\ldots+a_1x+a_0$ with scaling factor T}.
\end{definition}

Theorem~\ref{T:mt} implies that for an approximate well-covered independence polynomial $a_{\al}x^{\al}+\ldots+a_1x$, we have
\begin{equation}
	\frac{a_1}{\binom{\al}{1}}\le\frac{a_2}{\binom{\al}2}
	\le\ldots\le\frac{a_{\al}}{\binom{\al}{\al}}.\label{eqn:binoms}
\end{equation}
We will show that given a sequence of non-negative real numbers $(a_1,a_2,\ldots,a_{\al})$ satisfying (\ref{eqn:binoms}), the polynomial $\sum_{i=1}^{\al} a_i x^i$ is an approximate well-covered independence polynomial.  In order to do this, we will construct well-covered graphs with independence sequence satisfying (\ref{eqn:eps}) for some real number $T$.  We will construct these graphs from graphs satisfying the following property.

\begin{definition}
For integers $0\le k<\al$ and $1\le m$, say that graph $G$ satisfies the
property $P(k,\al;m)$ if:
\begin{enumerate}
\item All maximal cliques in $G$ are of size $\al$,
\item Each clique of size $k+1$ in $G$ is contained in a unique maximal clique, and
\item Each clique of size $k$ in $G$ is contained in at least $m$ maximal cliques.
\end{enumerate}
\end{definition}

Note that if $G$ satisfies property $P(k,\al;m)$, then its complement is a well-covered graph with independence number $\al$.  It seems that graphs that satisfy property $P(k,\al;m)$ have not been studied up to this point, but they are related to the study of maximal-clique covers and partitions in graphs (see, e.g., \cite{PSW}).  A \emph{maximal-clique covering} of a graph $G$ is a set of maximal cliques in $G$ whose union contains each edge of $G$ at least once.  A maximal-clique covering in which every edge is in exactly one element of the covering is a \emph{maximal-clique partition}.  In our case, instead of covering edges, we are covering cliques of size $k+1$ with maximal cliques.  In addition to this, we are covering cliques of size $k$ with at least $m$ distinct maximal cliques.  Clique coverings have recently been found to have implications in design theory (see, e.g., \cite{BBCM}) and so graphs satisfying $P(k,\al;m)$ may as well.

In Section~\ref{S:construction} we give a construction of graphs which
satisfy property $P(k,\al;m)$.  In Section~\ref{S:proofs}, we use these graphs to
prove that (\ref{eqn:binoms}) is a necessary condition for $a_{\al}x^{\al}+\cdots+a_1x$
to be an approximate well-covered independence polynomial.  Finally,
in Section~\ref{S:proofs2}, we show that this implies the Roller Coaster Conjecture, i.e., we prove the following.

\begin{thm}\label{thm:rct}
	Given a positive integer $\al$ and a permutation $\pi$ of $\set{\ceil{\al/2},\ceil{\al/2}+1,\ldots,\al}$, there is a well-covered graph $G$ with $\alpha(G)=\al$ and
	\[
		i_{\pi(\ceil{\al/2})}(G)<i_{\pi(\ceil{\al/2}+1)}(G)<\cdots<i_{\pi(\al)}(G).
	\]
\end{thm}

\section{Graph Construction}\label{S:construction} 

For a set $S$ and positive integer $k$, let $\binom Sk=\setof{A\of S}{\abs{S}=k}$.  Fix integers $k$, $\al$, and $m$ with $1\le k<\al$ and $1\le m$.  For $i\in [\al]$, define $\Fmk{i}$ to be the following set of functions:
\[
	\Fmk{i}=\set{f:\binom{[\al]\setminus \set{i}}{k}\to [m]}.
\]
Our graph is defined in terms of elements of $\Fmk{i}$.

\begin{definition}
	For integers $k$, $\al$, and $m$ with $1\le k<\al$ and $1\le m$, we define $H_{k,\al;m}$ to be the graph with vertex set 
	\[
		\bigcup_{i=1}^{\al} \Fmk{i},
	\]
	and, for $f\in \Fmk{i}$ and $g\in \Fmk{j}$, we let $f\sim g$ if and only if $i\neq j$ and 
	\[
		f\big|_A=g\big|_A,
	\]
	where $A=\binom{[\al]\wo \set{i,j}}k$.  If $k=0$, we define $H_{0,q;m}=mK_q$.
\end{definition}

For example, if $m=1$ and $k\ge 1$, then $\Fmk{i}$ consists of one (constant) function and so $H_{k,\al;m}=K_{\al}$.  For a function $f:\binom{[\al]}k\to[m]$, denote by $C_f$ the set of restrictions of $f$ to $\binom{[\al]\setminus\{i\}}k$ for $1\le i\le \al$. Note that $C_f$ has size $\al$.

\begin{lem}\label{lem:fcns}
For integers $k$, $\al$, and $m$ with $1\le k<\al$ and $1\le m$, every clique in $H_{k,\al;m}$ is contained in a clique of the form $C_f$ for some function $f:\binom{[\al]}k\to[m]$, and so each maximal clique in $H_{k,\al;m}$ is of size $\al$.  Furthermore, each clique of size $k+1$ is contained in a unique such clique, while every clique of size $k$ is contained in $m$ distinct such cliques.
\end{lem}

\begin{proof}
	In order for a set of vertices in $H_{k,\al;m}$, say $\set{f_1,f_2,\ldots,f_r}$, to be a clique, it must be the case that, for each $j\in [r]$, there is an $i_j\in [\al]$ such that $f_j\in \Fmk{i_j}$.  Further, we have $i_j\neq i_k$ if $j\neq k$.  If $A\in \binom{[\al]}k$ and $i_j\not\in A$ for some $j\in [r]$, then $A$ is in the domain of $f_j$ and any other function in the clique must agree with $f_j$ on $A$ (provided $A$ is in its domain).  Thus, any clique in $H_{k,\al;m}$ consists of restrictions of functions of the form $f:\binom{[\al]}k\to [m]$.  Note that if $B\in \binom{[\al]}k$ and $B\supseteq \setof{i_j}{j\in [r]}$, then $B$ is not in the domain of any of the functions in the clique.  
	
	Consider a clique in $H_{k,\al;m}$ of size $k+1$ consisting of vertices $\set{g_1,g_2,\ldots,g_{k+1}}$ where $g_j\in \Fmk{i_j}$ for $j\in [k+1]$.  There is no $A\in \binom{[\al]}k$ such that $A\supseteq \setof{i_j}{j\in [k+1]}$.  Thus, there is a unique $f:\binom{[\al]}k\to [m]$ such that $g_j\in C_f$ for each $j\in [k+1]$ and so there is a unique $\al$-clique containing the $(k+1)$-clique. 
	
	If $\set{h_1,h_2,\ldots,h_k}$ is a $k$-clique in $H_{k,\al;m}$ where $h_j\in \Fmk{i_j}$ for $j\in [k]$, then $B=\setof{i_j}{j\in [k]}\in \binom{[\al]}k$ is not in the domain of any of the $h_j$s.  Since a value on $B$ has not been specified, there are $m$ functions $f:\binom{[\al]}k\to [m]$ such that $h_j\in C_f$ for all $j\in [k]$.  Therefore, the $k$-clique is contained in at least $m$ maximal cliques in $H_{k,\al;m}$.
\end{proof}

\begin{thm}\label{C:bigguess}
For any integers $k$, $\al$, and $m$ with $0\le k<\al$ and $1\le m$, the graph $H_{k,\al;m}$ satisfies property $P(k,\al;m)$.
\end{thm}

\begin{proof}
	The case when $k\ge 1$ immediately from Lemma~\ref{lem:fcns}.  When $k=0$, we have $H_{0,\al;m}=mK_q$.  Each vertex, or $K_1$, in $mK_q$ is in a unique $K_q$, while the empty set is in all $m$ of the $K_q$s.  Thus, $H_{0,\al;m}$ satisfies $P(0,\al;m)$.
\end{proof}


\section{Partial converse to Theorem~\ref{T:mt}}\label{S:proofs} 

Our main goal in this section is to prove the following theorem.

\begin{thm}\label{thm:converse}
For all positive integers $\al$ and sequences of real numbers $(a_1, \ldots, a_{\al})$ satisfying
\[\frac{a_1}{\binom{\al}{1}}\le\frac{a_2}{\binom{\al}2}\le\ldots\le\frac{a_{\al}}{\binom{\al}{\al}},\]
$a_1x+a_2x^2+\ldots+a_{\al}x^{\al}$ is an approximate well-covered independence polynomial.
\end{thm}

We begin by using the graphs $H_{k,\al;m}$ to generate approximate well-covered independence polynomials.  

\begin{lem}\label{lem:hpoly}
For all integers $0\le k< \al$, the polynomial $\sum_{j=k+1}^{\al}\binom{\al}{j}x^j$ is an approximate well-covered independence polynomial.
\end{lem}

\begin{proof}
Fix $\epsilon>0$, and let $m$ be a positive integer such that $\frac{2^{\al}}m<\epsilon$.  Note that, by Theorem~\ref{C:bigguess}, for integers $k$, $\al$, and $m$ with $0\le k< \al$ and $1\le m$, $H_{k,\al;m}$ satisfies property $P(k,\al;m)$ and so the complement of $H_{k,\al;m}$ is a well-covered graph with independence number $\al$.

Suppose $H_{k,\al;m}$ has $T$ cliques of size $\al$, and let us consider the number of cliques of size $j$ for $1\le j\le \al$.  Clearly, there are $T\binom{\al}{j}$ pairs of cliques $(K_1,K_2)$ such that $K_1$ is of size $\al$, $K_2$ of size $j$, and $K_2\subseteq K_1$.  If $j\geq k+1$, then each clique of size $j$ contains a clique of size $k+1$, and
hence is contained in at most one clique of size $\al$. Since all maximal cliques
are of size $\al$, each clique of size $j$ is contained in a unique clique of size
$\al$, and hence there are $T\binom{\al}{j}$ cliques of size $j$. On the other hand, if $1\leq j\le k$, then each clique of size $j$ is contained in a clique of size $k$ and is therefore contained in at least $m$ cliques of size $\al$. Hence there are at most $T\binom{\al}{j}/m<T\epsilon$ cliques of size $j$.

Thus, if $G$ is the complement of $H_{k,\al;m}$, then $\frac{i_j(G)}T=\binom{n}{j}$ 
for $j\geq k+1$.  For $1\leq j\le k$, we have $\abs{\frac{i_j(G)}T-0}<\epsilon$, so $G$ is an
$\epsilon$-certificate for $\sum_{j=k+1}^n\binom{n}{j}x^j$ with scaling factor $T$.
\end{proof}

Now we show that the class of approximate well-covered independence polynomials
of a given degree is additive.  In order to do this, we use the join\footnote{The \emph{join} of graphs $G$ and $H$, denoted $G\vee H$, is the graph with vertex set $V(G)\cup V(H)$ and edge set $E(G)\cup E(H)\cup \setof{xy}{x\in V(G), y\in V(H)}$.} operation on graphs.  Note that if $G$ and $H$ are graphs and $k\geq 1$, then $i_k(G\vee H)=i_k(G)+i_k(H)$.

\begin{lem}\label{lem:sumpoly}
If $P_1(x)$ and $P_2(x)$ are approximate well-covered independence polynomials of degree $\al$,
then $P_1(x)+P_2(x)$ is an approximate well-covered independence polynomial of degree $\al$.
\end{lem}

\begin{proof}
Fix $\epsilon>0$. Let $G_1, G_2$ be $\frac{\epsilon}3$-certificates of $P_1(x)$ and 
$P_2(x)$ with scaling factors $T_1$ and $T_2$ respectively.  Suppose that all coefficients of $P_1(x)$ and $P_2(x)$ are bounded above by $N$, and let $k_1, k_2$ be positive integers such that 
\[
	1-\frac{\min(k_1T_1, k_2T_2)}{\max(k_1T_1, k_2T_2)}<\frac{\epsilon}{6N}.
\] 
Define $T:=\max(k_1T_1, k_2T_2)$.

Let $G$ be the graph defined as the join of $k_1$ copies of $G_1$ joined to the join of $k_2$ copies of $G_2$, i.e., 
\[
	G=\left(\bigvee_{i=1}^{k_1} G_1\right)\vee\left(\bigvee_{i=1}^{k_2} G_2\right).
\]
All independent sets in $G$ are completely contained in a single copy of $G_1$ or a single copy of $G_2$.  As such, $G$ is well-covered (since $\al(G_1)=\al(G_2)=\al$) and $i_j(G)=k_1i_j(G_1)+k_2i_j(G_2)$ for all $j\geq 1$.

Suppose that, for $j\geq 1$, the $x^j$ coefficients of $P_1(x)$ and $P_2(x)$ are $p^1_{j}$ and $p^2_{j}$, respectively. Then since $G_1$ is a $\frac{\epsilon}3$-certificate of $P_1(x)$ with scaling factor $T_1$, by definition, $\abs{p^1_{j}-\frac{i_j(G_1)}{T_1}}<\frac{\epsilon}3$.  Thus, since $k_1T_1\le T$, we have that 
\[
	\abs{\frac{k_1T_1p^1_{j}}T-\frac{k_1i_j(G_1)}T}<\frac{\epsilon}3.
\]	
Further, $p^1_{j}<N$ and 
\[
	1-\frac{k_1T_1}T<1-\frac{\min(k_1T_1, k_2T_2)}{\max(k_1T_1, k_2T_2)}<\frac{\epsilon}{6N},
\]
and so we have $\abs{p^1_{j}-\frac{k_1T_1p^1_{j}}T}<\frac{\epsilon}6$.  Thus, we see that
\[
	\abs{p^1_{j}-\frac{k_1i_j(G_1)}T}<\frac{\epsilon}2.
\]
Similarly,
\[
	\abs{p^2_{j}-\frac{k_2i_j(G_2)}T}<\frac{\epsilon}2.
\]
It follows that 
\[
	\abs{p^1_{j}+p^2_{j}-\frac{k_1i_j(G_1)+k_2i_j(G_2)}T}=\abs{p^1_{j}+p^2_{j}-\frac{i_j(G)}T}<
	\epsilon,
\]
so $G$ is an $\epsilon$-certificate of $P_1(x)+P_2(x)$ with scaling factor $T$.
\end{proof}

The same is true for more complicated linear combinations.

\begin{lem}\label{lem:polylincomb}
If $P_1(x), P_2(x), \ldots, P_k(x)$ are approximate well-covered independence polynomials of degree $\al$, and $\lambda_1, \ldots, \lambda_k$ are positive real numbers then
$\sum_{i=1}^k\lambda_iP_i(x)$ is an approximate well-covered independence polynomial of degree $\al$.
\end{lem}

\begin{proof}
If $G$ is an $\epsilon$-certificate of $P_i(G)$ with scaling factor $T$, then
it is a $\lambda_i\epsilon$-certificate of $\lambda_iP_i(G)$ with scaling factor $\frac{T}{\lambda_i}$.  Thus for each $i\in [k]$, $\lambda_iP_i(G)$ is an
approximate well-covered independence polynomial of degree $\al$, and therefore
the sum of them is by Lemma~\ref{lem:sumpoly}.
\end{proof}

With these lemmas in hand, we are now ready to prove the main result of this section, i.e., Theorem~\ref{thm:converse}.

\begin{proof}[Proof of Theorem~\ref{thm:converse}]
Fix a sequence $a_1,a_2,\ldots,a_{\al}$ satisfying 
\[
	\frac{a_1}{\binom{\al}{1}}\le\frac{a_2}{\binom{\al}{2}}\le\ldots\le\frac{a_{\al}}{\binom{\al}{\al}}.
\]
Let $b_1=\frac{a_1}{\binom{\al}{1}}$ and, for $i>1$, let $b_i=\frac{a_i}{\binom{\al}{i}}-\frac{a_{i-1}}{\binom{\al}{i-1}}$. Then, for all $i$, $b_i>0$ and 
\[
	a_i=\binom{\al}i\sum_{j=1}^i b_j.
\]
Let $P_k(x)=\sum_{j=k}^{\al}\binom{\al}{j}x^j$ so, by Lemma~\ref{lem:hpoly}, 
$P_k(x)$ is an approximate well-covered independence polynomial for all $1\le k\le \al$.  Therefore, by Lemma~\ref{lem:polylincomb}, so is 
\[
\sum_{j=1}^{\al} b_jP_j(x)=\sum_{j=1}^{\al} b_j\sum_{i=j}^{\al} \binom{\al}{i}x^i=\sum_{i=1}^{\al} \left(\sum_{j=1}^i\binom{\al}j b_j\right) x^i=\sum_{i=1}^{\al} a_ix^i.\qedhere
\]
\end{proof}


\section{Proof of the Roller Coaster Conjecture}\label{S:proofs2} 

Finally, we will show that Theorem~\ref{thm:converse} itself implies the Roller Coaster Conjecture.

\begin{lem}\label{lem:approxwell}
If $a_{\al}x^{\al}+\ldots+a_1x$ is an approximate well-covered independence polynomial
and $S$ is a subset of $[\al]$ such that $a_i\neq a_j$ if $i\neq j$ and $i,j\in S$,
then there exists a well-covered graph $G$ of independence number $\al$ such that
for all $j, k\in S$, $i_j(G)<i_k(G)$ if and only if $a_j<a_k$.
\end{lem}

\begin{proof}
Let 
\[
	\epsilon=\frac{1}3 \min\setof{\abs{a_i-a_j}}{i\neq j, i,j\in S}.
\]
Note that $\epsilon>0$. Let $G$ be an $\epsilon$-certificate of $a_{\al}x^{\al}+\ldots+a_1x+a_0$ with scaling factor $T$. Then $G$ is a well-covered graph of independence number $\al$ such that for all $j$, $|\frac{i_j(G)}T-a_j|<\epsilon$.

For $j, k\in S$, if $a_j<a_k$, we have $3\epsilon\le a_k-a_j$.  Thus, $a_j+\epsilon<a_k-\epsilon$, and so 
\[
	i_j(G)<T(a_j+\epsilon)<T(a_k-\epsilon)<i_k(G).\qedhere
\]
\end{proof}

Therefore, if we can any-order the initial coefficients of approximate well-covered
independence polynomials, we can do the same for actual well-covered independence polynomials.

\begin{lem}\label{lem:approxrct}
For any integer $n$ and for any permutation $\pi$ of the set
$\set{\ceil{\al/2},\ceil{\al/2}+1,\ldots,\al}$, there exists an approximate well-covered independence polynomial $a_{\al}x^{\al}+\cdots+a_1x+a_0$ such that for all $\ceil{\frac{\al}2}\le k,l\le \al$, $a_k<a_l$ if and only
if $\pi(k)<\pi(l)$.
\end{lem}

\begin{proof}
Define the sequence $(a_1,a_2,\ldots,a_{\al})$ as follows.
\[
	a_i=\begin{cases}
		\binom{\al}i & \text{if $1\le i<\ceil{\frac{\al}2}$},\\
		2^{\al}+\pi(i) & \text{if $\ceil{\frac{\al}2}\leq i\leq \al$}.
	\end{cases}
\]
Then $\frac{a_i}{\binom{\al}i}=1$ for $1\le i<\ceil{\frac{\al}2}$, while $\frac{a_i}{\binom{\al}i}>1$ for $\ceil{\frac{\al}2}\le i$. Further, for $\ceil{\frac{\al}2}\le i<\al$, 
\[
	\frac{a_i}{a_{i+1}}\le\frac{2^{\al}+\al}{2^{\al}}\le1+\frac{2}{\al},
\] 
while 
\[
	\frac{\binom{\al}i}{\binom{\al}{i+1}}=\frac{i+1}{\al-i}\ge\frac{\frac{\al}2+1}{\frac{\al}2}=1+\frac{2}{\al}.
\]
It follows that
\[\frac{a_1}{\binom{\al}{1}}\le\frac{a_2}{\binom{\al}{2}}\le\ldots\le\frac{a_{\al}}{\binom{\al}{\al}}.\]  
Therefore, by Theorem~\ref{thm:converse}, $a_{\al}x^{\al}+\ldots+a_1x$ is an approximate well-covered independence polynomial.  Furthermore, for $\ceil{\frac{\al}2}\le k,l\le \al$, $a_k=2^{\al}+\pi(k)<a_l=2^{\al}+\pi(l)$ if and only if
$\pi(k)<\pi(l)$.
\end{proof}

Our main theorem, Theorem~\ref{thm:rct}, follows.

\begin{proof}[Proof of Theorem~\ref{thm:rct}]
	The statement follows from applying Lemma~\ref{lem:approxrct} and then Lemma~\ref{lem:approxwell} with $S=\set{\ceil{\al/2},\ceil{\al/2}+1,\ldots,\al}$.
\end{proof}


\section{Conclusion}\label{S:conclusion}

Many interesting questions about the independence sequence of graphs are still open.  It was conjectured by Levit and Mandrescu \cite{LM} that every K\"onig-Egerv\'ary graph (a graph $G$ with $\alpha(G)+\nu(G)=n(G)$, where $\nu(G)$ is the size of the largest matching in $G$ and $n(G)$ is the number of vertices in $G$) has a unimodal independence sequence.  This conjecture was recently disproved by Bhattacharyya and Kahn \cite{BK}, who provided a bipartite graph with non-unimodal independence sequence (since every bipartite graph is a K\"onig-Egerv\'ary graph).  However, the following conjecture of Alavi et al. is still open.

\begin{conj}[Alavi, Erd\H{o}s, Madle, Schwenk \cite{AEMS}]
	Every tree and forest has unimodal independence sequence.
\end{conj}

We also believe that graphs satisfying property $P(k,\al;m)$ may be of independent interest.  Often the question for such structures is how small can such an object be?  To be precise, our question is as follows.

\begin{question}
	Given integers $k$, $\al$, and $m$ with $0\leq k<\al$ and $m\geq 1$, what is the minimum number of vertices in a graph $G$ with property $P(k,\al;m)$?
\end{question}

The graph $H_{k,\al;m}$ has $\al m^{\binom{\al-1}{k}}$ vertices which we suspect is far from the minimum.  Recall that, for integers $k$ and $n$ with $1\leq k\leq n$, the Kneser graph $KG_{n,k}$ is the graph with vertex set $\binom{[n]}k$, where two vertices are adjacent if and only if they are disjoint.  One can check that $KG_{\al(\al-2),\al-2}$ satisfies property $P(\al-2,\al;\frac{1}2\binom{2(\al-2)}{\al-2})$.  Further, we have
\[
	n(H_{\al-2,\al;m})=\al m^{\al-1}\gg \binom{\al(\al-2)}{\al-2}=n(KG_{\al(\al-2),\al-2})
\]
when $m=\frac{1}2\binom{2(q-2)}{q-2}$.

\bibliographystyle{amsplain}
\bibliography{rollercoaster}

\end{document}